\documentclass[11pt]{amsart}

\usepackage{geometry}
\usepackage{todonotes}

\usepackage{amsfonts, amssymb, amscd}
\numberwithin{equation}{section}

\usepackage[symbol]{footmisc}

\usepackage{bm}
\usepackage{verbatim}
\usepackage{mathrsfs}
\usepackage{graphicx}
\usepackage{tikz-cd}
\usepackage{subcaption}
\usepackage{listings}
\usepackage{subfiles}
\usepackage[toc,page]{appendix}
\usepackage{mathtools}
\usepackage{comment}
\usepackage{enumerate}
\usepackage{enumitem}

\usepackage{graphicx}
\graphicspath{{images/}}

\usepackage{appendix}
\usepackage{hyperref}
\newcommand{\bb}{\bm{b}}

\newcommand{\Mm}{{\bf{M}}}

\newcommand{\Pp}{\mathbb{P}}
\newcommand{\Qq}{\mathbb{Q}}

\newcommand{\Rr}{\mathbb{R}}

\newcommand{\Center}{\operatorname{center}}

\newcommand{\Exc}{\operatorname{Exc}}

\newcommand{\glct}{\operatorname{glct}}

\newcommand{\Supp}{\operatorname{Supp}}

\newcommand{\mult}{\operatorname{mult}}

\newcommand{\Dd}{\mathcal{D}}

\newcommand{\Oo}{\mathcal{O}}

\newtheorem{thm}{Theorem}[section]

\newtheorem{lem}[thm]{Lemma}

\newtheorem{exprop}[thm]{Example-Proposition}

\theoremstyle{definition}
\newtheorem{defn}[thm]{Definition}
\newtheorem{ques}[thm]{Question}
\theoremstyle{definition}
\newtheorem{rem}[thm]{Remark}

\newtheorem{ex}[thm]{Example}

\theoremstyle{definition}

\begin{document}

\title{Number of singular points on projective surfaces}
\author{Jihao Liu and Lingyao Xie}

\address{Department of Mathematics, The University of Utah, Salt Lake City, UT 84112, USA}
\email{lingyao@math.utah.edu}

\address{Department of Mathematics, The University of Utah, Salt Lake City, UT 84112, USA}
\email{jliu@math.utah.edu}

\date{\today}

\begin{abstract}
The number of singular points on a klt Fano surface $X$ is $\leq 2\rho(X)+2$.
\end{abstract}

\maketitle

\tableofcontents

\section{Introduction}

We work over the field of complex numbers $\mathbb C$. For any normal projective variety $X$, we let $\rho(X)$ be the Picard number of $X$.

Let $X$ be a klt Fano surface, i.e. a klt projective surface such that $-K_X$ is ample. It is interesting to ask when are the number of singular points of $X$ bounded from above, and to give an estimate of the maximal number of singular points on $X$.

For simplicity, for any surface $X$, we let $n(X)$ be the number of singular points on $X$. When $X$ is klt Fano, Keel and M\textsuperscript{c}Kernan show that $n(X)\leq 5$ when $\rho(X)=1$ \cite[Page 72]{KM99}. This is strengthened by Belousov who shows that $n(X)\leq 4$:
\begin{thm}[{\cite[Theorem 1.2]{Bel08},\cite[Theorem 1.1]{Bel09}}]\label{thm: four singular point}
Let $X$ be a klt Fano surface such that $\rho(X)=1$. Then $n(X)\leq 4$.
\end{thm}
This bound is optimal even for Fano surfaces with canonical singularities by \cite{MZ88} (see also \cite{Fur86,Zha87,Zha88,Zha89} and Example \ref{ex: lc is necessary}(1)). In this note, we show that $n(X)$ is bounded from above by a number depending only on $\rho(X)$.
\begin{thm}\label{thm: number sing surface is 2rho+2}
Let $X$ be a klt Fano surface. Then $n(X)\leq 2\rho(X)+2$.
\end{thm}
It is easy to see that Theorem \ref{thm: number sing surface is 2rho+2} and Theorem \ref{thm: four singular point} are equivalent when $\rho(X)=1$.

In fact, we can relax the assumption ``klt Fano" to ``$(X,B)$ is klt log Calabi-Yau for some boundary $B\not=0$" without changing the bound $2\rho(X)+2$. Moreover, we can relax the assumption ``klt Fano" to ``$(X,B)$ is lc and $-(K_X+B)$ is nef for some boundary $B$" if we allow a small increase on the bound $2\rho(X)+2$. We have the following result:
\begin{thm}\label{thm:General Calabi Yau Case}
Let $(X,B)$ be an lc surface pair such that $-(K_X+B)$ is nef. Then:
\begin{enumerate}
    \item $n(X)\leq \max\{2\rho(X)+10,16\}$.
    \item If $(X,B)$ is klt, $B\not=0$, and $K_X+B\equiv 0$, then $n(X)\leq 2\rho(X)+2$.
    \item If $X$ is klt and $K_X\not\equiv 0$, then $n(X)\leq 2\rho(X)+4$.
    \item If $X$ is klt but not canonical and $K_X\equiv 0$, then $n(X)\leq 2\rho(X)+7$.
    \item If $X$ is canonical and $K_X\equiv 0$, then $n(X)\leq 16$.
    \item If $X$ is not klt, then $n(X)\leq 2\rho(X)+10$.
    \item If $X$ is not klt and $-K_X$ is big and nef, then $n(X)\leq 2\rho(X)+7$.
\end{enumerate}
\end{thm}

\begin{rem}
\begin{enumerate}
\item The assumption of Theorem \ref{thm:General Calabi Yau Case}(2) includes the case when $X$ is klt Fano or of Fano type, hence immediately implies Theorem \ref{thm: number sing surface is 2rho+2}.
\item Theorem \ref{thm: number sing surface is 2rho+2} may be well-known to experts, but we cannot find any references except \cite{Bel08,Bel09,KM99}, and we cannot find any similar results in papers citing \cite{Bel08}, \cite{Bel09}, or \cite{KM99}, so we believe that Theorem \ref{thm: number sing surface is 2rho+2} is new.
    \item The assumption ``$-(K_X+B)$ is nef" in Theorem \ref{thm:General Calabi Yau Case} cannot be further relaxed to ``$-K_X$ is pseudo-effective" even when $X$ is canonical and $-K_X$ is effective (Example-Proposition \ref{exm:2rho+2 is optimal for rho=2}(1)).
    \item The assumption ``$(X,B)$ is lc" Theorem \ref{thm:General Calabi Yau Case} cannot be further relaxed even when $\rho(X)=1$ and $X$ is Fano, otherwise $n(X)$ may be unbounded (Example \ref{ex: lc is necessary}(3)).
    \item The bounds for Theorem \ref{thm:General Calabi Yau Case}(2)(3) are optimal at least for low Picard numbers and the bounds for Theorem \ref{thm:General Calabi Yau Case}(5) is optimal. We don't know if the bounds for Theorem \ref{thm:General Calabi Yau Case}(4)(6) are optimal even for small values of $\rho(X)$, however $2\rho(X)+2$ is not satisfied even when $\rho(X)=1$ and $X$ is Fano (Example \ref{ex: lc is necessary}(2)). 
    \item We expect some boundedness results on singular points to hold in high dimensions (see Section 5). We prove the boundedness on the number of torus invariant singular points for toric varieties with bounded Picard numbers (Theorem \ref{thm: torid bdd number sing}), but one needs to be careful for non-toric varieties due to Example-Proposition \ref{exprop: ex threefold unbounded isolated singularities}.
\end{enumerate}
\end{rem}

\noindent\textbf{Acknowledgement}.  The authors would like to thank Christopher D. Hacon for useful discussions and suggestions. The authors would like to thank Paolo Cascini, Meng Chen, Jingjun Han, Chen Jiang, Yuchen Liu, and Qingyuan Xue for useful discussions. The authors were partially supported by NSF research grants no: DMS-1801851, DMS-1952522 and by a grant from the Simons Foundation; Award Number: 256202.

\section{Preliminaries}

We adopt the standard notation and definitions in \cite{KM98} and \cite{BCHM10}.

\subsection{Pairs and singularities}

\begin{defn}\label{defn: positivity}
	A \emph{pair} $(X,B)$ consists of a normal quasi-projective variety $X$ and an $\Rr$-divisor $B\ge0$ such that $K_X+B$ is $\Rr$-Cartier. If $B\in [0,1]$, then $B$ is called a \emph{boundary}.
	
	Let $E$ be a prime divisor on $X$ and $D$ an $\mathbb R$-divisor on $X$. We define $\mult_{E}D$ to be the \emph{multiplicity} of $E$ along $D$. Let $\phi:W\to X$
	be any log resolution of $(X,B)$ and let
	$$K_W+B_W:=\phi^{*}(K_X+B).$$
	The \emph{log discrepancy} of a prime divisor $D$ on $W$ with respect to $(X,B)$ is $1-\mult_{D}B_W$ and it is denoted by $a(D,X,B).$ We say that $(X,B)$ is lc (resp. klt) if $a(D,X,B)\geq 0$ (resp. $>0$) for every log resolution $\phi:W\to X$ as above and every prime divisor $D$ on $W$. 
	
	A germ $X\ni x$ consists of a normal quasi-projective variety $X$ and a closed point $x\in X$.
\end{defn}

\begin{defn}
    Let $f: X\dashrightarrow Y$ be a birational map which does not extract any divisor, $p: W\rightarrow X$ and $q: W\rightarrow Y$ a common resolution, and $D$ an $\Rr$-Cartier $\Rr$-divisor on $X$ such that $f_*D$ is $\Rr$-Cartier. We say that $f$ is \emph{$D$-negative} if
    $$p^*D=q^*D_Y+E$$
    for some $E\geq 0$, and $\Exc(f)\subset\Supp(p_*E)$.
\end{defn}

\begin{defn}\label{defn: fano type}
Let $X$ be a normal projective variety. We say that $X$ is \emph{Fano} if $-K_X$ is ample. We say that $X$ is \emph{of Fano type} if $(X,B)$ is klt and $-(K_X+B)$ is ample for some boundary $B$ on $X$.  We say that $(X,B)$ is log Calabi-Yau if $K_X+B\equiv 0$.
\end{defn}

\subsection{Surfaces}

\begin{defn}
A surface is a normal quasi-projective variety of dimension $2$. For any non-negative integer $m$, the Hirzebruch surface $\mathbb F_m$ is given by $\mathbb P_{\mathbb P^1}(\mathcal{O}_{\mathbb P^1}\oplus\mathcal{O}_{\mathbb P^1}(m))$.

 In some references, a klt Fano surface is also called a \emph{log del Pezzo} surface.
\end{defn}

\begin{defn}[Dual graph]\label{defn: crt}
Let $n$ be a non-negative integer, and $C=\cup_{i=1}^nC_i$ a collection of irreducible curves on a smooth surface $U$. We define the \emph{dual graph} $\Dd(C)$ of $C$ as follows.
\begin{enumerate}
    \item The vertices $v_i=v_i(C_i)$ of $\Dd(C)$ correspond to the curves $C_i$.
    \item For $i\neq j$,the vertices $v_i$ and $v_j$ are connected by $C_i\cdot C_j$ edges.
\end{enumerate}

For any birational morphism $f: Y\rightarrow X$ between surfaces, let $E=\cup_{i=1}^nE_i$ be the reduced exceptional divisor for some non-negative integer $n$. We define $\Dd(f):=\Dd(E)$. 

A dual graph is called a \emph{tree} if the graph contains no cycles. 
\end{defn}

\begin{lem}\label{lem: certain blow up tree}
\begin{enumerate}
    \item Let $f': Y'\rightarrow X\ni x$ be a resolution of a klt surface germ $X\ni x$. Then $\Dd(f')$ is a tree whose vertices are all smooth rational curves.
    \item Let $f': Y'\rightarrow X$ be a projective morphism between smooth surfaces. Then $\Dd(f')$ is a tree whose vertices are all smooth rational curves.

\end{enumerate}
\end{lem}
\begin{proof}
 (1) follows from Lemma \cite[Lemma 3.10]{HL20} and the classification of klt surface singularities by taking $f: Y\rightarrow X$ to be the minimal resolution of $X\ni x$. (2) follows from (1) because $Y'$ is a resolution of $X$.
\end{proof}

\begin{lem}\label{lem: lc surface pair q gorenstein}
Let $(X,B)$ be an lc surface pair. Then $K_X$ is $\Qq$-Cartier.
\end{lem}
\begin{proof}
Pick any closed point $x\in X$. If $(X,0)$ is numerically dlt near $x$, then $X$ is $\Qq$-Cartier near $x$ by \cite[Proposition 4.11]{KM98}. If $(X,0)$ is not numerically dlt near $x$, since $(X,B)$ is lc, $(X,B)$ is numerically lc near $x$. By \cite[Corollary 4.2]{KM98}, $x\not\in B$, hence $K_X$ is $\Qq$-Cartier near $x$. Thus $K_X$ is $\Qq$-Cartier.
\end{proof}

\begin{lem}\label{lem: MMP get f1 surface}
Let $X$ be a smooth projective surface, $X\rightarrow Z\cong\mathbb P^1$ a fibration, and $X\rightarrow X'$ a $K_X$-MMP over $Z$ which terminates with a Mori fiber space $f: X'\rightarrow Z$. Assume that $X\not=X'\cong\mathbb P^1\times\mathbb P^1$. Then there exists a $K_X$-MMP over $Z$: $X\rightarrow X''$, such that $X''\cong\mathbb F_1$. 
\end{lem}
\begin{proof}
Let $g: Y\rightarrow X'$ be the last step of the MMP $X\rightarrow X'$. Then $g$ is a blow-up of a closed point $x'\in X'$. Let $F:=g^{-1}_*(f^{-1}(f(x')))$, then $F^2=-1$. Let $g': Y\rightarrow X''$ be the contraction of $F$, then $X''\cong\mathbb F_1$ and the induced morphism $X\rightarrow X''$ is a $K_X$-MMP.
\end{proof}

\subsection{G-pairs}

We need the following definitions on generalized pairs (g-pairs for short). See \cite{BZ16} for more details.

\begin{defn}[$\bb$-divisors]\label{defn: b divisors} Let $X$ be a normal quasi-projective variety. We call $Y$ a \emph{birational model} over $X$ if there exists a projective birational morphism $Y\to X$. 

Let $X\dashrightarrow X'$ be a birational map. For any valuation $\nu$ over $X$, we define $\nu_{X'}$ to be the center of $\nu$ on $X'$. A \emph{$\bb$-divisor} $\Mm$ over $X$ is a formal sum $\Mm=\sum_{\nu} r_{\nu}\nu$ where $\nu$ are valuations over $X$, such that $\nu_X$ is not a divisor except for finitely many $\nu$. If in addition, $r_{\nu}\in\Qq$ for every $\nu$, then $\Mm$ is called a \emph{$\Qq$-$\bb$-divisor}. The \emph{trace} of $\Mm$ on $X'$ is the $\Rr$-divisor
$$\Mm_{X'}:=\sum_{\nu_{i,X'}\text{ is a divisor}}r_i\nu_{i,X'}.$$
If $\Mm_{X'}$ is $\Rr$-Cartier and $\Mm_{Y}$ is the pullback of $\Mm_{X'}$ on $Y$ for any birational model $Y$ of $X'$, we say that $\Mm$ \emph{descends} to $X'$, and write $\Mm=\overline{\Mm_{X'}}$. If $X$ is projective and $\Mm$ is a $\bb$-divisor over $X$, such that $\Mm$ descends to some birational model $Y$ over $X$ and $\Mm_Y$ is nef, then we say that $\Mm$ is \emph{nef}.
\end{defn}

\begin{defn}[G-pairs]\label{defn: gpairs}
	A \emph{projective g-pair} $(X,B,\Mm)$ consists of a normal projective variety $X$, an $\Rr$-divisor $B\geq 0$ on $X$, and a nef $\bb$-divisor $\Mm$ over $X$, such that $K_X+B+\Mm_X$ is $\Rr$-Cartier. If $B$ is a $\Qq$-divisor and $\Mm$ is a $\Qq$-$\bb$-divisor, then we say that $(X,B,\Mm)$ is a $\Qq$-g-pair.
	
	Let $(X,B,\Mm)$ be a projective g-pair, $\phi:W\to X$ any log resolution of $(X,\Supp B)$ such that $\Mm$ descends to $W$, and
	$$K_W+B_W+\Mm_W:=\phi^{*}(K_X+B+\Mm_X).$$
	We say that $(X,B,\Mm)$ is \emph{glc} if the coefficients of $B_W$ are $\leq 1$.
	
	For any projective glc g-pair $(X,B,\Mm)$ and $\Rr$-Cartier $\Rr$-divisor $D\geq 0$ on $X$, we define
	$$\glct(X,B,\Mm;D):=\sup\{t\mid (X,B+tD;\Mm)\text{ is glc}\}$$
to be the \emph{glc threshold} of $D$ with respect to $(X,B,\Mm)$.
\end{defn}

\section{Proofs of the main theorems}

\begin{lem}\label{lem: antinef imply antinef q}
Let $(X,B)$ be an lc pair such that $(X,B)$ is lc and $-(K_X+B)$ is nef (resp. $K_X+B\equiv 0$). Then there exists a $\Qq$-divisor $B'$ on $X$ such that $(X,B')$ is lc and $-(K_X+B')$ is nef (resp. $K_X+B\equiv 0$).
\end{lem}
\begin{proof}
Cf. \cite[Proposition 2.6]{HL20a}, \cite[Corollary 3.5]{HL19}, and \cite[Lemma 5.4,Theorem 5.6]{HLS19}.
\end{proof}

\begin{lem}\label{lem: divisorial contraction reduce sing by 2}
Let $X$ be a klt surface, $f: X\rightarrow Y$ a $K_X$-negative divisorial contraction of a curve $C$, and $y:=f(C)$. Then $C$ contains at most $2$ singular points of $X$.
\end{lem}
\begin{proof}
We may assume that $C$ contains $n$ singular points of $X$ for some integer $n\geq 3$. Let $g: W\rightarrow X$ be the minimal resolution of $X$ near $C$ with exceptional divisors $E_1,\dots,E_m$ for some integer $m\geq n$. Let $C_W:=g^{-1}_*C$. Possibly reordering indices, we may assume that $C_W$ intersects $E_1,E_2$ and $E_3$.

Since $f\circ g$ is a resolution of $Y\ni y$, by Lemma \ref{lem: certain blow up tree}(1), $C_W\cong\mathbb P^1$. If $C_W\le-2$ then $f\circ g$ is actually the minimal resolution of $Y\ni y$. But $a(C,Y,0)>1$ since $f$ is $K_X$-negative, thus $C$ is not contained in the minimal resolution of $Y\ni y$. Hence $C_W^2=-1$ and we may let $p: W\rightarrow T$ be the contraction of $C_W$. Then there exists an induced morphism $h: T\rightarrow Y$ which is a resolution of $Y\ni y$. Let $E_{i,T}:=p_*E_i$ for each $i$, then $E_{i,T}\cdot E_{j,T}\ge1$ for every $i,j\in\{1,2,3\}$ with $i\not=j$. Thus $\Dd(\cup_{i=1}^mE_{i,T})=\Dd(h)$ is not a tree, which contradicts Lemma \ref{lem: certain blow up tree}(1).


\end{proof}

\begin{lem}\label{lem: fibers containing singular point is not Cartier}
Let $X$ be a klt surface, $f: X\to Z$ a $K_X$-Mori fiber space such that $\dim Z=1$, and $z\in Z$ a closed point. If  $f^*z$ is reduced, then $X$ is smooth near $f^{-1}z$. 
\end{lem}
\begin{proof}
Since $f:X\to Z$ is a $K_X$-Mori fiber space, $f^{-1}z$ is an irreducible curve and $R^if_*\Oo_X=0$ for any $i>0$. Since $Z$ is regular and $X$ is Cohen-Macaulay, $f$ is flat \cite[Theorem 23.1]{Mat89}. 
If $f^*z$ is reduced, then by Cohomology and Base change \cite[III 12.11]{Har77}, $H^1(X_z,\Oo_{X_z})=0$ so $X_z\cong \Pp^1$. Combining with the fact that $f$ is flat, we deduce that $X$ is regular along $f^{-1}z$ because both $X_z=f^{-1}z$ and $Z$ are regular \cite[Theorem.23.7]{Mat89}.
\end{proof}

\begin{lem}\label{lem: number singlarity picard 2 case}
Let $(X,B)$ be an lc projective surface pair such that $-(K_X+B)$ is nef, and $f: X\rightarrow Z$ a $K_X$-Mori fiber space such that $\dim Z=1$. Then
\begin{enumerate}
    \item any fiber of $f$ contains at most $2$ singular points of $X$,
    \item 
    \begin{enumerate}
        \item at most four fibers of $f$ contain singular point(s) of $X$, and
        \item if $(X,B)$ is klt and $K_X+B\equiv 0$, then at most three fibers of $f$ contain singular point(s) of $X$.
    \end{enumerate}
\end{enumerate}
\end{lem}
\begin{proof}
By Lemma \ref{lem: antinef imply antinef q}, we may assume that $B$ is a $\Qq$-divisor. There exists a non-negative integer $n$, closed points $z_1,\dots,z_n\in Z$, and fibers $F_i:=f^{-1}z_i$ for each $i$, such that $F_1,\dots,F_n$ are the only closed fibers of $f$ which contain singular points of $X$. If $n=0$, there is nothing left to prove, so in the rest of the proof, we may assume that $n\geq 1$.

First we prove (1). Suppose that there exists a fiber $F$ of $f$, such that $F$ contains at least $3$ singular points of $X$ and $F=f^{-1}z$ for some closed point $z\in Z$. We let $g: W\rightarrow X$ be the minimal resolution of $X$, $E_1,\dots,E_m$ the $g$-exceptional divisors for some integer $m\geq 3$ such that $\Center_XE_i\in F$ for each $i$, and $F_W:=g^{-1}_*F$. Then $E_i^2\leq -2$ for each $i$. Possibly reordering indices, we may assume that $F_W$ intersects $E_1,E_2,E_3$.

We may run a $K_W$-MMP over $Z$, which induces a birational contraction $h: W\rightarrow Y$ between smooth projective varieties and a $K_Y$-Mori fiber space $f': Y\rightarrow Z$, such that $Y$ is a geometrically ruled surface. In particular, $h$ contracts $m$ elements of $\{F_W,E_1,\dots,E_m\}$. Since $Y$ is smooth and $X$ is not smooth, $F_W$ is contracted by $h$. Since $W$ is smooth, $h$ is a $K_W$-MMP over $Z$, and $E_i^2\leq -2$, we have that $F_W\cong\mathbb P^1$ and $F_W^2=-1$. Thus we may let $p: W\rightarrow T$ be the contraction of $F_W$, and there is an induced morphism $q: T\rightarrow Y$. Let $E_{i,T}:=p_*E_i$ for each $i$, then $E_{i,T}\cdot E_{j,T}\ge1$ for every $i,j\in\{1,2,3\}$ with $i\not=j$. Thus $\Dd(\cup_{i=1}^mE_{i,T})$ is not a tree, hence $\Dd(q)$ is not a tree, which contradicts Lemma \ref{lem: certain blow up tree}(2).

\medskip

Now we prove (2.a). We let $\Mm_X:=-(K_X+B)$ and $\Mm:=\overline{\Mm_X}$. Then $(X,B,\Mm)$ is a projective glc $\Qq$-g-pair. By the generalized canonical bundle formula (\cite[Theorem 1.4]{Fil18}, \cite[Theorem 1.2]{HL19}), we have
$$0=K_X+B+\Mm_X\sim_{\mathbb Q}f^*(K_Z+B_Z+M_Z)$$
such that $M_Z$ is pseudo-effective and $$\mult_{z}B_Z=1-\glct(X,B,\Mm;f^*z)$$
for any point $z\in Z$. By Lemma \ref{lem: fibers containing singular point is not Cartier}, each $f^*z_i$ is not reduced, hence $\mult_{z_i}B_Z\geq\frac{1}{2}$ for each $i$. Thus 
$$0=\deg(K_Z+B_Z+M_Z)\ge -2+n\cdot\frac{1}{2}+0=-2+\frac{n}{2},$$
which implies that $n\le 4$. Since $n\ge 1$, we have $\deg(K_Z)<0$ so $Z\cong\Pp^1$. Moreover, $n=4$ if and only if $M_Z\sim_{\mathbb Q}0$ and $B_Z=\frac{1}{2}\sum_{i=1}^4z_i$.

\medskip

Under the assumptions of (2.b), by abundance, $K_X+B\sim_{\mathbb Q}0$, so $$K_X+B\sim_{\mathbb Q}f^*(K_Z+B_Z+M_Z)$$
is the canonical bundle formula for $K_X+B$ as well. Let $g: Y\rightarrow X$ be the minimal resolution of $X$ and $h: X\rightarrow X'$ a $K_Y$-MMP over $Z$ which terminates with a Mori fiber space $f': X'\rightarrow Z$. By Lemma \ref{lem: MMP get f1 surface}, we may assume that $X'\cong\mathbb F_m$ for some positive integer $m$. Let $K_{X'}+B':=h_*g^*(K_X+B)$, then $$K_{X'}+B'\sim_{\mathbb Q}(f')^*(K_Z+B_Z+M_Z)$$
is the canonical bundle formula for $K_{X'}+B'$, and $(X',B')$ is klt. 

Assume that $n=4$. Then $M_Z\sim_{\mathbb Q}0$ and $B_Z=\frac{1}{2}\sum_{i=1}^4z_i$. Let $\tau$ and $\sigma$ denote the fiber and the negative section of $\mathbb F_m$, and $D$ the only element of $|\sigma|$. Then $-K_{X'}\sim 2\sigma+(m+2)\tau$. By the definition of the canonical bundle formula, $B'\geq (f')^*B_Z\sim_{\mathbb Q}2\tau$, which implies that $0\leq B'-(f')^*B_Z\in |2\sigma+m\tau|_{\mathbb Q}$. Thus for any element $D'\in |2\sigma+m\tau|_{\mathbb Q}$, $D'\geq D$ (cf. \cite[Chapter 5, Proposition 3]{Fri98}). So $B'\geq D$, hence $(X',B')$ is not klt, a contradiction.
\end{proof}

\begin{proof}[Proof of Theorem \ref{thm:General Calabi Yau Case} (2)(3)]
Since $K_X$ is not pseudo-effective, we may run a $K_X$-MMP which terminates with a Mori fiber space $f: Y\rightarrow Z$. Let $g: X\rightarrow Y$ be the induced morphism and $B_Y:=g_*B$, then $-(K_Y+B_Y)$ is nef. Moreover, if $(X,B)$ is klt log Calabi-Yau, then $(Y,B_Y)$ is klt log Calabi-Yau.
\medskip

\noindent\textbf{Case 1}. $\dim Z=0$. In this case, $\rho(Y)=1$ and $Y$ is klt Fano, so $f$ is a composition of $\rho(X)-1$ divisorial contractions between klt surfaces. By Lemma \ref{lem: divisorial contraction reduce sing by 2} and Theorem \ref{thm: four singular point},
$$n(X)\leq n(Y)+2(\rho-1)\leq 4+2(\rho-1)=2\rho+2.$$ 

\medskip

\noindent\textbf{Case 2}. $\dim Z=1$. In this case, $\rho(Y)=2$, so $f$ is a composition of $\rho(X)-2$ divisorial contractions between klt surfaces. By Lemma \ref{lem: divisorial contraction reduce sing by 2}, $n(X)\leq n(Y)+2(\rho-2)$. By Lemma \ref{lem: number singlarity picard 2 case}, $n(Y)\leq 8$ and $n(Y)\leq 6$ when $(X,B)$ is klt log Calabi-Yau. Thus $n(X)\leq 2\rho(X)+4$ and $n(X)\leq 2\rho(X)+2$ when $(X,B)$ is klt log Calabi-Yau and $B\not=0$.
\end{proof}








\begin{proof}[Proof of Theorem \ref{thm:General Calabi Yau Case}(4)]
Since $X$ is klt but not canonical and $K_X\equiv0$, there exists an extraction $f:Y\to X$ of a prime divisor $E$ such that $Y$ is klt and $K_Y+aE=f^*K_X\equiv0$ for some positive real number $a$. By Theorem \ref{thm:General Calabi Yau Case}(3),
$$n(Y)\leq 2\rho(Y)+4=2\rho(X)+6,$$ 
thus $n(X)\leq n(Y)+1\leq 2\rho(X)+7$.
\end{proof}

\begin{proof}[Proof of Theorem \ref{thm:General Calabi Yau Case}(5)]
By abundance, $K_X\sim_{\mathbb Q}0$, hence there exists a smallest positive integer $m$ such that $mK_X\sim0$. Since $K_X$ is Cartier, there exists an \'{e}tale cyclic cover $Y\to X$ of degree $m$ such that $K_Y\sim 0$. In particular, $Y$ is canonical and $n(X)\leq n(Y)$ (cf. \cite[Lemma 2.51]{KM98}). 

Let $f: W\rightarrow X$ be the minimal resolution of $Y$. Then $K_W=f^*K_Y\sim 0$, hence $W$ is either an abelian surface or a smooth K3 surface. If $W$ is an abelian surface, then $W$ does not contain any rational curves, so $W=Y$ and hence $n(Y)=0$. If $W$ is a smooth K3 surface, then $Y$ is a K3 surface with at most canonical singularities. By \cite[Corollary 4.6]{Pet20}, $n(Y)\leq 16$. Thus $n(X)\leq n(Y)\leq 16$.
\end{proof}



\begin{proof}[Proof of Theorem \ref{thm:General Calabi Yau Case}(6)(7)]
By Lemma \ref{lem: lc surface pair q gorenstein}, $K_X$ is $\Qq$-Cartier. By (3), we may assume that $X$ is not klt, hence there exists at least $1$ point on $X$ where $X$ is not klt.
 By applying the connectedness theorem (\cite[Proposition 3.3.2]{Pro01},\cite[Theorem 1.2]{HH19},\cite[Theorem 1.2(1)]{Bir20}) to $(X,B)$ (or apply \cite[Theorem 1.1]{FS20} to the g-pair $(X,B,\Mm:=\overline{-(K_X+B)})$; see also \cite[Lemma 6.9]{Sho92}), we know that there exist at most 2 points on $X$ where $X$ is not klt. If $-K_X$ is big and nef, then by the Shokurov-Koll\'ar connectedness principle, there exists exactly 1 point on $X$ where $X$ is not klt.
 
 Thus there exists an extraction $f: Y\rightarrow X$ and a divisor $E\geq 0$ on $X$, such that $Y$ is klt, $1\leq\rho(Y)-\rho(X)\leq2$, $K_Y+E=f^*(K_X+B)$, $(Y,E)$ is lc, and $-(K_Y+E)$ is nef. Moreover, $\rho(Y)-\rho(X)=1$ when $-K_X$ is big and nef. In particular, $K_Y\not\equiv 0.$ By Theorem \ref{thm:General Calabi Yau Case}(3), $n(Y)\leq 2\rho(Y)+4$, hence $n(X)\leq n(Y)+2\leq2\rho(X)+10$, and $n(X)\leq n(Y)+1\leq2\rho(X)+7$ when $-K_X$ is big and nef.
\end{proof}

\begin{proof}[Proof of Theorem \ref{thm:General Calabi Yau Case}]
We only left to prove (1), which follows from (3)(4)(5)(6).
\end{proof}

\begin{proof}[Proof of Theorem \ref{thm: number sing surface is 2rho+2}]
It follows from Theorem \ref{thm:General Calabi Yau Case}(2).
\end{proof}



\section{Examples on surfaces}

In this section, we discuss how far our bounds in Theorem \ref{thm:General Calabi Yau Case} are away from being optimal. The following Example-Proposition shows that even when $\rho(X)=2$,
\begin{enumerate}
    \item the assumption ``$-(K_X+B)$ is nef" is necessary in Theorem \ref{thm:General Calabi Yau Case},
    \item Theorem \ref{thm:General Calabi Yau Case}(2) is optimal even when $X$ is klt Fano, and
    \item Theorem \ref{thm:General Calabi Yau Case}(3) is optimal.
\end{enumerate} 

\begin{exprop}\label{exm:2rho+2 is optimal for rho=2}
Let $n$ be a positive integer, $Z:=\mathbb P^1\times\mathbb P^1$, and $z_i:=(u_i,v_i)\in Z$ closed points in $Z$ for any $i\in\{1,2,\dots,n\}$ such that $u_i\not=u_j$ for any $i\not=j$. We let $p_1: Z\rightarrow\mathbb P^1$ and $p_2:Z\rightarrow\mathbb P^1$ are the first and second projection of $Z$ to $\Pp^1$, and $L_i:=p_1^*u_i$ and $R_i:=p_2^*v_i$ for each $i$.

We let $f: Y\rightarrow Z$ be the blow-up of $z_1,\dots,z_n$. For each $i$, we let $E_i$ be the exceptional curve of $f$ over $z_i$, $L_{i,Y}:=f^{-1}_*L_i$, $R_{i,Y}:=f^{-1}_*R_i$, and $y_i:=L_{i,Y}\cap E_i$. 
We let $g: X\rightarrow Y$ be the blow-up of $y_1,\dots,y_n$. For each $i$, we let $F_i$  be the exceptional curve of $g$ over $y_i$, $L_{i,X}:=g^{-1}_*L_{i,Y}$, $R_{i,X}:=g^{-1}_*R_{i,Y}$, and $E_{i,X}:=g^{-1}_*E_i$. 

We let $h: X\rightarrow S$ be the contraction of $E_{1,X},\dots,E_{n,X}$ and $L_{1,X},\dots,L_{n,X}$. For each $i$, we let $F_{i,S}:=h_*F_i$, $R_{i,S}:=h_*R_{i,X}$, $s_i:=h(E_{i,X})$, and $t_i:=h(L_{i,X})$. Then $s_1,\dots,s_n$ and $t_1,\dots,t_n$ are the only singular points on $S$ and are $\frac{1}{2}(1,1)$ singularities. 

\begin{enumerate}
\item When $t_1=t_2\dots=t_n$, $-K_S$ is effective.
\item When $n=4$ and $t_1=t_3\not=t_2=t_4$, $(S,B)$ is lc log Calabi-Yau for some $B$.
\item When $n=3$ and $t_i\not=t_j$ for any $i\not=j$, $S$ is klt Fano.
\end{enumerate}
\end{exprop}
\begin{proof}
Most of the proofs are are elementary computations on pullbacks and pushforwards of divisors which we omit. In (1), $-K_S\sim 4F_{1,S}+2R_{1,S}\geq 0$. In (2), we may pick $B=R_{1,S}+R_{2,S}$. In (3), $-K_S\sim 2R_{i,S}$, $R_{i,S}^2=\frac{1}{2}$, $s_i\in R_{i,S}$, and $t_i\in R_{j,S}$ for any $i\not=j$. Thus $-K_S$ is nef and big and we may let $\phi: S\rightarrow T$ be the ample model of $-K_S$.

If $S\not=T$, then $-K_S$ is not ample, and $\phi$ contracts an irreducible curve $C\subset S$ such that $-K_S\cdot C=0$. Since $\rho(S)=2$, $T$ is a klt Fano variety and $\rho(T)=1$. Since $-K_S\sim 2R_{i,S}$ for any $i$ and $R_{i,S}^2>0$, $C$ does not intersect $R_{i,S}$ for any $i$, so $C$ is contained in the smooth locus of $S$. Thus $n(T)\geq n(S)=6$, which contradicts Theorem \ref{thm: four singular point}. 

Thus $S=T$, hence $-K_S$ is ample, and we are done.
\end{proof}

The following example shows that even when $\rho(X)=1$ and $X$ is Fano,
\begin{enumerate}
    \item Theorem \ref{thm: four singular point} is optimal,
    \item the bound ``$2\rho(X)+2$" is not enough if $X$ is not klt, and
    \item the assumption ``$(X,B)$ is lc" is necessary for Theorem \ref{thm:General Calabi Yau Case}.
\end{enumerate} 

\begin{ex}\label{ex: lc is necessary}
Assumptions and notations as in Example-Proposition \ref{exm:2rho+2 is optimal for rho=2} and assume that $t_1=t_2\dots=t_n$. Let $R':=p_2^*v$ for some $v\not=v_1$, and $R'_{S}:=h_*((f\circ g)^{-1}_*R')$. Since the intersection matrix of $R_{1,X}\cup_{i=1}^n(E_{i,X}\cup L_{i,X})$ is negative definite, there exists a contraction $\phi: S\rightarrow T$ of $R_{1,S}$. In particular, $\rho(T)=1$. Since $$D:=-(K_S+\frac{2(n-2)}{n}R_{1,S})\sim 4F_{1,S}+\frac{4}{n}R_{1,S}$$
is big and nef and $\phi$-trivial, and since $nD\sim 4nF_{1,S}+4R_{1,S}\sim 4nF_{2,S}+4R_{1,S}\sim 4R'_{S}$, $|nD|$ is base-point-free and defines $\phi$. Thus $nD\sim\phi^*\phi_*(nD)$, and in particular, $-K_T=\phi_*D$ is ample. Since $a(R_{1,S},T,0)=\frac{4-n}{n}$, we have:
 \begin{enumerate}
     \item When $n=3$, $T$ is a klt Fano surface, $\rho(T)=1$ and $n(T)=4$. 
     \item When $n=4$, $T$ is an lc Fano surface, $\rho(T)=1$ and $n(T)=5>2\rho(T)+2$.
     \item When $n\geq 5$, $T$ is a non-lc Fano surface, $\rho(T)=1$ and $n(T)=n+1$. When $n\rightarrow+\infty$, $n(T)\rightarrow+\infty$.
 \end{enumerate}
\end{ex}

The following well-known example shows that Theorem \ref{thm:General Calabi Yau Case}(5) is optimal: 
\begin{ex}
Some Kummer surfaces are canonical K3 surfaces with 16 singular points.
\end{ex}

We do not know if Theorem \ref{thm:General Calabi Yau Case}(4)(6) are optimal or not even when $\rho(X)=1$, and we do not now if Theorem \ref{thm:General Calabi Yau Case}(2)(3) are optimal when $\rho(X)$ is large. We guess that under the assumption of Theorem \ref{thm:General Calabi Yau Case}, $n(X)\leq\rho(X)+C$ for some constant number $C$, but we do not know how to prove this yet. The next example shows that the linear term $\rho(X)$ is necessary in an expression of an upper bound of $n(X)$ even when $X$ is klt Fano.

\begin{ex}
Fix a positive integer $n\ge 2$, let $e_1=(1,0),~e_2=(0,1)\in \mathbb R^2$, and $u_{-1}=-e_1,~u_i=ie_1+(i^2-1)e_2~(0\le i\le n)$. Then each $u_i$ is primitive. Now let $\Sigma$ be the complete fan in $N_{\mathbb R}=\mathbb R^2$ generated by rays $u_{-1},u_0,\dots,u_n$. Then  the projective toric surface $X_{\Sigma}$ is klt Fano with $\rho(X_{\Sigma})=n+2-2=n$. The number of singular points corresponds to the number of non-smooth maximal cones in $\Sigma(2)=\{\text{Cone}(u_{i-1},u_i),\text{Cone}(u_n,u_{-1}))~|0\le i\le n\}$. Notice that $\text{Cone}(u_n,u_{-1}),\text{Cone}(u_{i-1},u_{i})$ ($2\leq i\leq n$) are not smooth because none of $\{u_n,u_{-1}\},\{u_{i-1},u_{i}\}$ ($2\leq i\leq n$) generates $N=\mathbb Z^2$. Thus $X_{\Sigma}$ has exactly $n$ singular points.
\end{ex}

\section{Discussions}

For toric varieties, the singular locus is torus invariant and thus can be nicely described as a disjoint union of torus orbits:
\begin{thm}\label{thm: torid bdd number sing}
If $X$ is a proper $\mathbb Q$-factorial toric variety of dimension $d$, then for any $2\le k\le d$, there exists a polynomial $h_k$ of degree $\le\min\{k, d-1\}$ such that the number of torus invariant singular points of codimension $k$ on $X$ is $\leq h_k(\rho(X))$. 
\end{thm}
\begin{proof}
Let $\Sigma$ be the complete fan in $N_{\mathbb R}\cong \mathbb R^d$ which defines $X$, then the cones in $\Sigma$ are all simplicial and naturally gives a triangulation of $S^{d-1}\cong\{\mathbb R^d-0\}/(x\sim \lambda x)$, where each cone of dimension $k\ge 1$ corresponds to a $(k-1)$-simplex. 

 Recall $\Sigma(k)$ is the set of $k$ dimensional cones in $\Sigma$, then we have $\rho(X)+d=|\Sigma(1)|$ and $|\Sigma(k)|\le\binom{|\Sigma(1)|}{k}$. Thus any $|\Sigma(k)|~(1\le k\le d-1)$ is bounded by a polynomial of $\rho(X)$ with degree $\le k$. Also, we have $1-(-1)^d=\chi(S^{d-1})=\sum_{k=1}^{d}(-1)^{k-1}|\Sigma(k)|$. Hence $|\Sigma(d)|$ is bounded by a polynomial of $\rho(X)$ with degree $\le d-1$. Since the torus invariant singular points correspond to torus orbits in Sing$(X)$, the statements follows directly by the orbit-cone correspondence theorem.
 
\end{proof}

It is natural to ask whether we can have a bound on the number of singular points in high dimensions for non-toric klt Fano varieties with bounded Picard number as well. However, the first question is: since the singular locus may be of dimension $>0$, how can we effectively define the ``number of singular points" for a non-toric variety?

The most straightforward idea is to consider the number of isolated singular points. Unfortunately, we have the following counterexample for klt Fano varieties with only isolated singularities of Picard number 1 even in dimension 3. This example is given by Chen Jiang:

\begin{exprop}\label{exprop: ex threefold unbounded isolated singularities}
Fix a positive integer $k$. Let $X=X_{6k+3}\subset\mathbb P(1,3,3,3k+1,3k+2)$ be a general hypersurface of degree $6k+3$. Then
\begin{enumerate}
\item $X$ is quasismooth klt Fano of Picard number $1$, and
\item $X$ contains exactly the following singularities:
\begin{enumerate}
\item a cyclic quotient singularity of type $\frac{1}{3k+1}(1,3,3)$,
\item a cyclic quotient singularity of type $\frac{1}{3k+2}(1,3,3)$, and
\item $(2k+1)$ cyclic quotient singularities of type $\frac{1}{3}(1,1,2)$.
\end{enumerate}
\end{enumerate}
\end{exprop}
\begin{proof}
(1) follows from \cite[Theorem 8.1]{IF00} (see also \cite[Theorem 2.7]{CJL20}) and \cite[Theorem 3.2.4(i)]{Dol82}. (2) follows from (1) and \cite[Section 9--10]{IF00} (see also \cite[Theorem 2.8]{CJL20}).
\end{proof}

Nevertheless, we may still ask the following questions. These questions arise in personal communications of the first author with Paolo Cascini, Christopher D. Hacon, Jingjun Han and Chen Jiang during the summer of 2020.

\begin{ques}
Let $d,\rho$ be two positive integers. Does there exist a positive integer $N_1=N_1(d,\rho)$, such that for any klt Fano variety $X$ of dimension $d$ with $\rho(X)\leq\rho$, the number of isolated non-terminal singularities of $X$ is $\leq N_1$?
\end{ques}

\begin{ques}
Let $d,\rho$ be two positive integers. Does there exist a positive integer $N_2=N_2(d,\rho)$, such that for any klt Fano variety $X$ of dimension $d$ with $\rho(X)\leq\rho$, the number of codimension $2$ singularities of $X$ is $\leq N_2$?
\end{ques}
Theorem \ref{thm: number sing surface is 2rho+2} answers these two questions when $d=2$, but both questions seem to be widely open in dimension $\geq 3$ even when $\rho=1$. We remark that if we have satisfactory answers for these questions in the Picard number 1 case, then the methods used in our paper are expected to be applied to prove the bounded Picard number cases.

For similar questions and results, we also refer the readers to \cite{BMSZ18}.

\end{document}